\documentclass{amsart}

\usepackage{amsmath,amsthm,amsfonts,amssymb,amscd,mathrsfs}
\usepackage[centertags]{amsmath}
\usepackage{psfrag,graphicx}
\usepackage{multirow}
\usepackage[all]{xy}

\usepackage{ulem}
\usepackage{remark}

\usepackage{epic,eepic}
\usepackage{color}
\usepackage{ebezier}
\usepackage{enumerate}
\usepackage{mathrsfs}
\usepackage{graphicx}

\theoremstyle{plain}

\newtheorem{Thm}{Theorem}[section]
\newtheorem{Lem}[Thm]{Lemma}
\newtheorem{Prop}[Thm]{Proposition}
\newtheorem{Cor}[Thm]{Corollary}
\theoremstyle{plain}
\newremark{Def}[Thm]{Definition}
\newremark{example}[Thm]{Example}
\newremark{Rem}[Thm]{Remark}
\newremark{Construction}[Thm]{Construction}
\newremark{Question}[Thm]{Question}

\numberwithin{equation}{section}

\newcommand{\Spec}{\mathrm{Spec}} 
\newcommand{\Len}{\mathrm{length}}
\newcommand{\Div}{\mathrm{div}}
\newcommand{\cO}{\mathcal{O}} 
\newcommand{\cL}{\mathcal{L}} 
\newcommand{\cI}{\mathcal{I}}
\newcommand{\ie}{{\it i.e.}, }

\begin{document}
\title{On Pluri-canonical Systems of Arithmetic Surfaces}
\author{Yi GU}
\address{
Peking University\\
5, Yiheyuan Road,
Beijing, China
}
\email{pkuguyi2010@gmail.com}

\address{
Institut de Math\'ematiques de Bordeaux \\
Universit\'{e} de Bordeaux\\
351, Cours de la Lib\'eration \\
33405 Talence, France 
}
\email{Yi.Gu@math.u-bordeaux1.fr} 

\begin{abstract}
Let $S$ be a Dedekind scheme with perfect residue fields at closed points, let $f: X\rightarrow S$ be a minimal regular arithmetic surface of fibre genus at least $2$ and let $f': X'\rightarrow S$ be the canonical model of $f$. It is well known that $\omega_{X'/S}$ is relatively ample. In this paper we prove that $\omega_{X'/S}^{\otimes n}$ is relative very ample for all $n\geq 3$.
\end{abstract}
\thanks{}
\maketitle

\begin{section}{Introduction}\label{intro} 
Let $S$ be a Dedekind scheme, let $f: X\rightarrow S$ be a minimal regular arithmetic surface of (fibre) genus at least $2$ (see \cite{LQ}, \S 9.3), and 
let $\omega_{X/S}$ be the relative dualizing invertible sheaf on $X$.
It is well known that $\omega_{X/S}^{\otimes n}$ is globally generated when $S$ is
affine and $n$ is 
big enough (actually $n\ge 2$ suffices, see \cite{L}, Theorem 7 and Remark~\ref{rmk} below). 
But in general $\omega_{X/S}$ is not relatively ample even when $X\to S$ is semi-stable, because of the possible 
presence of vertical $(-2)$-curves (\ie vertical curves $C$ in $X$ such that $\omega_{X/S}\cdot C=0$).
Let $f': X'\rightarrow S$ be the canonical model of $X$ 
obtained by contracting all vertical $(-2)$-curves (\cite{LQ}, \S 9.4.3). Then 
the relative dualizing sheaf $\omega_{X'/S}$ is a relatively ample invertible 
sheaf on $X'$ (op. cit., \S 9.4.20).

\begin{Question} For which $n\in \mathbb{N}$ is $\omega_{X'/S}^{\otimes n}$ relatively very ample ? 
\end{Question} 

When $f$ is smooth, then $X=X'$ and it is well known that $\omega_{X'/S}^{\otimes n}$ is relatively very ample for all $n\ge 3$. More generally, the same result holds if $f$ is semi-stable (hence $X'\to S$ is stable), see \cite{DM}, Corollary of Theorem 1.2. 
In this paper we show that the same bound holds in the general case under the mild assumption that the
residue fields of $S$ at closed points are perfect (we do not known whether this condition can
be removed). Notice that this condition is satisfied if $S$ is the spectrum of the ring of integers
of a number field $K$, or if $S$ is a  regular, integral quasi-projective curve over a perfect field. 

{\begin{Thm}[{\bf Main Theorem}]\label{main}
Let $S$ be a Dedekind scheme with perfect residue fields at closed points, let $f: X\rightarrow S$ be a minimal regular arithmetic surface of fibre genus at least $2$, and let $f': X'\rightarrow S$ be the canonical model of $f$. Then
$\omega_{X'/S}^{\otimes n}$ is relatively very ample for all $n\geq 3$. 
\end{Thm}}

\begin{Rem}\label{rmk} Under the above hypothesis and when $S$ is affine, 
we also give a new proof of a theorem of J. Lee (\cite{L}, Theorem 7):  
\par\medskip
 
$\omega_{X/S}^{\otimes n}$ 
(equivalently, $\omega_{X'/S}^{\otimes n}$) is globally generated for all $n\geq 2$. 
\par\medskip
\noindent The proof in \cite{L} is based on a detailed case-by-case analysis.
Our method is more synthetic. 
Moreover, we do not 
assume that the generic fibre of $f$ is geometrically irreducible, see
Proposition~\ref{Lee}. However, as a counterpart, we need the assumption on the
perfectness of the residue fields.
\end{Rem}

The proof of \ref{main} relies on the following criterion due to F. Catanese, M. Franciosi, K. Hulek and M. Reid.

\begin{Thm}[\cite{CFHR}, Thm. 1.1]\label{CR} Let $C$ be a
    Cohen-Macaulay curve over an algebraically closed field $k$,  let
    $H$ be an invertible sheaf on $C$. 
Then \begin{enumerate}
\item $H$ is globally generated if for every generically Gorenstein subcurve 
$B\subseteq C$, we have $H\cdot B\geq 2p_a(B)$;
\item $H$ is very ample if for every generically Gorenstein subcurve $B\subseteq C$, 
we have $H\cdot B\geq 2p_a(B)+1$. 
\end{enumerate}
\end{Thm}

Recall that the arithmetic genus is defined by $p_a(B):=1-\chi_k(\cO_B)$ and that 
$B$ is called \emph{generically Gorenstein} if its dualizing sheaf is generically invertible. 
We will apply the above theorem to the closed fibres of $f'$. The key point 
is Corollary~\ref{genus calculate} which allows to express 
the arithmetic genus of an effective Weil divisor $B$ on $X'$ as the arithmetic 
genus of some effective Cartier divisor on $X$, to overcome the absence of 
the adjunction formula. 

When $S$ is of equal-characteristic (\ie $\cO_S$ contains a field),  we provide
in Section \ref{proof 2} a simpler proof of the main theorem based on I. Reider's method. 

\

{\bf Acknowledgements} The author would like to thank Prof. Qing Liu, who 
suggested this problem to him and helped a lot to improve the presentation of 
this paper. The author would also like to thank Prof. Jinxing Cai and Jilong Tong for helpful suggestions and encouragements. 
\end{section}

\begin{section}{Arithmetic genus of Weil divisors}\label{section genus}

In this section we show how to compute the arithmetic genus of an effective Weil divisor on a normal surface with rational singularities. The main result is Corollary~\ref{genus calculate}. 

Let $Y$ be a noetherian normal scheme.  
We first associate to each Weil divisor $D\in Z^1(Y)$ a coherent sheaf $\cO_Y(D)\subseteq \mathcal{K}$, where $\mathcal{K}$ is the constant sheaf $K(Y)$ of rational functions on $Y$. 

\begin{Def} 
Let $D\in Z^1(Y)$. For any point $\xi\in Y$ of codimension $1$, denote by $v_\xi$ the 
normalised discrete valuation on $K(Y)$ induced by $\xi$ and by $v_\xi(D)\in\mathbb Z$ the multiplicity of $D$
at $\overline{\{\xi\}}$. We define the sheaf $\cO_Y(D)$ 
by $$\cO_Y(D)(U)=\{x\in K(Y) \mid v_\xi(x)+v_\xi(D)\geq 0, \quad \forall  \xi\in U, \ \mathrm{codim}(\xi,Y)=1\}.$$
(By convention, $v_\xi(0)=+\infty$).
\end{Def}

\begin{Rem}\label{Id} 
\begin{enumerate}
\item Any $\cO_Y(D)$ is a reflexive sheaf by \cite{H}, Prop.1.6 since it is a normal sheaf. 
\item For any effective Weil divisor $0<D\in Z^1(Y)$, $\cO_Y(-D)$,  also denoted by $\cI_D$, is an ideal sheaf. In this case, we  consider $D$ as \emph{the subscheme of $Y$ defined $\cI_D$}. 
\item A purely codimension one closed subscheme $Z$ of $Y$ is an effective Weil divisor if and only if it has no embedded points. 
\end{enumerate}
\end{Rem}

\begin{Lem}\label{saturated} Suppose $D$ is an effective Weil divisor on $Y$. 
Let $Y^0$ be an open subset of $Y$ with complement of codimension at least $2$, and let $Z$ be any closed subscheme of $Y$ such that  $D\cap Y^0=Z\cap Y^0$. 
Then $\cI_Z\subseteq \cI_D$. 
\end{Lem}

\begin{proof}
Let $j: Y^0\rightarrow Y$ be the canonical morphism.  We have $j^*\cI_Z=j^*\cI_D$ by  hypothesis. Now the canonical injective morphism  $\cI_Z\hookrightarrow j_*j^*\cI_Z=j_*j^*\cI_D$ implies 
$\cI_Z\subseteq \cI_D$ because $j_*j^*\cI_D=\cI_D$ (\cite{H}, Prop.1.6). 
\end{proof}

From now on we focus on the case of surfaces. More 
precisely, we suppose $Y=\mathrm{Spec}(R)$ is an affine noetherian normal surface (\ie of dimension $2$), with a unique singularity $y\in Y$. Suppose that $Y$ admits a desingularization $g:T\rightarrow Y$ (\ie $T$ is  regular, $g$ is proper birational and $g:T\setminus g^{-1}(y)\rightarrow Y\setminus \{y\}$ is an isomorphism). Let $\{E_i\}_{i=1,...,m}$ be the set of prime divisors contained in $g^{-1}(y)$. We call an \emph{effective} divisor $Z$ on $X$ \emph{exceptional} if $g(Z)=\{y\}$ as sets, \ie $Z=\sum r_iE_i\geq 0$.

\begin{Def} For any exceptional divisor $Z$ and any coherent sheaf $\mathcal{F}$ on $Z$, we define 
the  characteristic of $\mathcal{F}$: 
$$\chi(\mathcal{F}):=\sum \limits_{i\ge 0}  (-1)^i \Len_{R} H^i(Z,\mathcal{F}).$$ 
We call $p_a(Z):=1-\chi(\cO_Z)$ the \emph{arithmetic genus of $Z$}.
\end{Def}

This definition makes sense since $H^i(Z,\mathcal{F})$ is a finitely generated $R$-module supported on $y$, hence of finite length, and $H^i(Z,\mathcal{F})=0$ as soon as $i\geq 2$  since $\dim Z\le 1$. 
\begin{Rem}
In case $T$ is an open  subscheme of an arithmetic surface, and $Z$ is supported in a special fibre, this definition coincides with \cite{LQ}, page 431, Equality (4.12). 
\end{Rem}
\medskip

We recall the definition of rational singularities for surfaces due to M. Artin : 

\begin{Prop}[\cite{Ar}, see also \cite{LQ}, \S 9.4]\label{rational}
The following conditions are equivalent:
\begin{enumerate}
\item for any exceptional divisor $Z$, we have $H^1(Z,\cO_Z)=0$;
\item for any $Z$ as above, $p_a(Z)\leq 0$;
\item $R^1g_*(\cO_T)=0$.
\end{enumerate}
When any of the above condition holds, $y$ is called a \emph{rational singularity}. 
\end{Prop}

\begin{Rem}
It is well known that the canonical model $X'$ in Section 1 has only rational singularities (\cite{LQ}, Corollary 9.4.7).  
\end{Rem}

From now on we assume \emph{$y$ is a rational singularity.} 
We need an important lemma. 

\begin{Lem}[\cite{Lip}, Theorem 12.1]\label{global generated}
For any $\cL\in \mathrm{Pic}(T)$ such that $\deg (\cL|_{E_i})\geq 0$
for all $i\le m$, 
 $\cL$ is globally generated and $H^1(T,\cL)=0$. \qed
\end{Lem}

\begin{Prop}\label{genus} Let $B$ be an effective Weil divisor on $Y$, and let $\widetilde{B}$ be its strict transform in $T$. Then 
there is an exceptional divisor $D$  on $T$ such that $\cI_B\cO_T=\cO_T(-\widetilde{B}-D)$,  $g_*\cO_T(-\widetilde{B}-D)=\cI_B$ and $R^1g_*\cO_T(-\widetilde{B}-D)=0$.
\end{Prop}

\begin{proof} Let $\cI:=\cI_B\cO_T$. Let $\Lambda$ be the set of exceptional divisors $Z$ such that 
$$(\widetilde{B}+Z)\cdot E_i\leq 0, \quad \forall i\leq m.$$ 
Exactly as in \cite{Ba}, Lemma 3.18, one can 
prove that there is a smallest element $D$ in $\Lambda$. 

Let $x\in \cI_B\setminus \{ 0\}$. Then $\Div_T(x)=\widetilde{B}+\widetilde{C}+Z$, where $Z$ is exceptional and 
$\widetilde{C}$ is an effective divisor which does not contain any exceptional divisor. 
As $\Div_T(x)\cdot E_i=0$, we have $(\widetilde{B}+Z)\cdot E_i=-E_i\cdot \widetilde{C}\leq 0$. So by definition $Z\in \Lambda$ and $\widetilde{B}+Z\geq \widetilde{B}+D$, hence $\cI\subseteq \cO_T(-\widetilde{B}-D)$.  
Since $g_*(\cO_T(-\widetilde{B}-D))|_{Y\setminus \{y\}}= \cI_B|_{Y\setminus \{y\}}$, it follows from 
Lemma~\ref{saturated} that $g_*\cO_T(-\widetilde{B}-D)\subseteq \cI_B$. 
Hence $$\cI_B\subseteq g_*\cI\subseteq g_*\cO_T(-\widetilde{B}-D)\subseteq \cI_B.$$
Therefore  $g_*\cO_T(-\widetilde{B}-D)=\cI_B$. By Lemma \ref{global generated}, $g^*g_*\cO_T(-\widetilde{B}-D)\twoheadrightarrow \cO_T(-\widetilde{B}-D)$, we have $\cI=\cO_T(-\widetilde{B}-D)$. The same lemma implies that
$R^1g_*\cO_T(-\widetilde{B}-D)=H^1(T, \cO_T(-\widetilde{B}-D))=0$.
\end{proof}

In the next corollary, we use the term ``exceptional divisor on $X$''in the sense of exceptional divisors for $X\to X'$.

\begin{Cor}\label{genus calculate}
\begin{enumerate}
\item Let $X, X', S$ be as in Section~\ref{intro}.  Let $B$ be an effective vertical Weil divisor on $X'$. 
Then there is an exceptional divisor $D_B$ on $X$ 
such that $p_a(\widetilde{B}+D_B)=p_a(B)$. 

\item If $M'$ is a normal proper algebraic surface over a field and with at worst rational singularities, and if $g: M\rightarrow M'$ 
is any resolution of singularities, then for any effective Weil divisor $B$ on $M'$, there is an exceptional divisor $D_B$ such that $p_a(\widetilde{B}+D_B)=p_a(B)$. 
\end{enumerate}
\end{Cor}

\begin{proof} 
(1) Let $g: X\to X'$ be the canonical map. Let $\cI_B$ be the ideal sheaf on $X'$ defined as 
in Remark~\ref{Id} (3) and let $\widetilde{B}$ be the strict transform of $B$ in $X$.
By Proposition \ref{genus}, there is an exceptional divisor $D_B$ on $X$ 
such that $g_*\cO_X(-\widetilde{B}-D_B)=\cI_B$ and $R^1g_*\cO_X(-\widetilde{B}-D_B)=0$. Consider the exact sequence:
$$0\rightarrow \cO_X(-\widetilde{B}-D_B)\rightarrow \cO_X\rightarrow \cO_{\widetilde{B}+D_B}\rightarrow 0.$$
Push-forwarding  by $g_*$, we get:
$$0\rightarrow \cI_B\rightarrow\cO_{X'}\rightarrow g_*\cO_{\widetilde{B}+D_B}\rightarrow R^1g_*\cO_X(-\widetilde{B}-D_B)=0,$$
and 
$$0=R^1g_*\cO_X\rightarrow R^1g_*\cO_{\widetilde{B}+D_B}\rightarrow R^2g_*\cO_X(-\widetilde{B}-D_B)=0.$$
So $g_*\cO_{\widetilde{B}+D_B}=\cO_B$ and $R^1g_*\cO_{\widetilde{B}+D_B}=0$.
Hence $\chi(\cO_{\widetilde{B}+D_B})=\chi(\cO_B),$ and $p_a(B)=p_a(\widetilde{B}+D_B)$.

(2) The proof is the same as for (1).
\end{proof}
\end{section}

\section{proof of the main theorem}\label{proof 1}
Now we can prove the main theorem. We keep the notation of 
section $\ref{intro}$.

\begin{Lem}\label{Weil}
For any vertical subcurve $B_1$ of $X'$, there is a vertical effective 
Weil divisor $B$ on $X'$ and an open subset $U'\subseteq X'$ 
with complement of codimension at least $2$, such that $B\cap U'=B_1\cap U'$ and $p_a(B)\geq p_a(B_1)$. 
\end{Lem}

\begin{proof} The curve 
$B_1$ has only finitely many embedded points. Let 
$U'$ be the complement of these points. Now as $B_1\cap U$ do not admit any embedded points, it is an effective Weil divisor, and extends to a unique effective Weil divisor $B$ on $X'$. By Lemma \ref{saturated}, $\cI_{B_1}\subseteq \cI_B$ and, since the cokernel is clearly a skyscraper sheaf, we have $\chi(\cO_B)-\chi(\cO_{B_1})=-\chi(\cI_B/\cI_{B_1})\leq 0$, hence $p_a(B)\geq p_a(B_1)$.  
\end{proof}

Since in our main theorem we do not assume that the generic fibre is
geometrically  connected, we need some more
preparation. Let 
$$\xymatrix{
X\ar[r]^{f}\ar[dr]^{g}& S\\
& S'\ar[u]_{\pi}}
$$
be the Stein factorization of $f$. For any $s\in S$, 
each connected component of $X_{s}$ is equal set-theoretically 
to some fibers of $X\to S'$. 
The following proposition is well known: 

\begin{Prop}\label{general type does not depend on stein factorization}
\begin{enumerate}
\item The dualizing sheaf $\omega_{S'/S}$ is invertible, 
 and we have $\omega_{X/S}=g^*\omega_{S'/S}\otimes \omega_{X/{S'}}$;
\item The arithmetic genus of (the generic fiber of)
  $f$ is larger than $1$ if and only if that of $g$ is larger than
  $1$. \qed
\end{enumerate}
\end{Prop}

\begin{Lem}[\cite{L}, Lemma 4]\label{special} Let
    $s\in S$ be a closed point. Let $\cL\in \mathrm{Pic}(X)$. Then
$H^{1}(X_{s}, \cL|_{X_{s}})=0$ if for any effective divisor $0<A\leq
X_{s}$, we have $\cL\cdot A>(\omega_{X/S}+A)\cdot A$. In particular, 
$H^{1}(X_{s}, \omega_{X_{s}/k(s)}^{\otimes n})=0$ whenever $n\geq 2$. 
\end{Lem}

\begin{proof} The first part of
the lemma is proved in \cite{L}. We only have to 
show that $\cL:=\omega^{\otimes n}_{X/S}$ satisfies the required
condition for any $n\ge 2$. If $\omega_{X/S}^{\otimes n}\cdot A\leq
(\omega_{X/S}+A)\cdot A$, then $\omega_{X/S}\cdot A-A^2\leq 0$ 
because $\omega_{X/S}\cdot A\ge 0$ by the minimality
  of $X\to S$, hence $\omega_{X/S}\cdot A=0, A^2=0$, in particular
$A$ is the sum of some multiples of fibres of 
$g$ (see the proof of \cite{LQ}, Theorem 9.1.23).  
But $\omega_{X/S}\cdot X_{s'}>0$ for any closed point $s'\in S'$ by Proposition \ref{general type does not depend on stein factorization}. Contradiction. 
\end{proof}
\medskip

\begin{proof}[Proof of the  main theorem]
First we observe that  the
 statement
is local on $S$, so we assume $S$ is local. As the constructions of $X$
and of $X'$ commute with \'etale base changes, replacing $S$ by its
strict henselisation if necessary, we can assume that the residue
field $k$ at the closed point $s\in S$ is algebraically closed. 
Here we use the hypothesis that $k(s)$ is perfect.

By Lemma \ref{special} and the upper-semicontinuity theorem
(\cite{H1}, III.12), for any $n\geq 2$, we have 
$R^1f_*(\omega_{X/S}^{\otimes n})=0$. So  
$R^1f'_*(\omega_{X'/S}^{\otimes n})=0$ since 
$\pi_{*}\omega_{X/S}^{\otimes n}=\omega_{X'/S}^{\otimes n}$ and 
$$R^1\pi_{*}\omega_{X/S}^{\otimes n}
=R^1\pi_*(\pi^{*}\omega_{X'/S}^{\otimes n})=
R^1\pi_*\cO_X\otimes_{\cO_{X'}}\omega_{X'/S}^{\otimes n}=0.$$ 
Again by the upper-semicontinuity theorem 
we have $H^1(X'_{s},\omega_{X'/S}^{\otimes n})=0$, 
and the canonical morphism $H^0(X', \omega_{X'/S}^{\otimes n})\otimes k\rightarrow
H^0(X'_s,\omega_{X'_s/k}^{\otimes n})$ is an isomorphism.

Let $n\ge 3$. Then it is enough to prove that 
$\omega_{X'/S}^{\otimes n}|_{X'_s}=\omega_{X'_s/k}^{\otimes n}$ is very ample.  Let
$C=X'_s$.  This is a
Cohen-Macaulay curve over $k$. We will show for any subcurve
$B_1$ of $C$ that 
$$n (\omega_{X'_s/k}\cdot B_1)\geq 3(\omega_{X'_s/k}\cdot B_1) \geq
2p_a(B_1)+1.$$ 
Theorem~\ref{CR} (2) will then imply that $\omega_{X'_s/k}^{\otimes n}$ is very ample.
Let $B$ be the effective Weil divisor on $X'$ as given by 
Lemma \ref{Weil}. By construction, $p_a(B)\ge p_a(B_1)$.  
As $B$ and $B_1$ differ at worst by a zero-dimensional closed subset,
we have $\omega_{X'_s/k}\cdot B=\omega_{X'_s/k}\cdot B_1$ 
(use {\it e.g.} \cite{GLL}, Prop. 6.2). Therefore it is enough to show 
the desired inequality for $B$. 

By Corollary \ref{genus calculate}, $p_a(B)=p_a(\widetilde{B}+D)$ for
some exceptional divisor $D$. Note in our case that 
$\omega_{X'_s/k}\cdot B= \omega_{X'/S}\cdot B=\omega_{X/S}\cdot \widetilde{B}=\omega_{X/S}\cdot (\widetilde{B}+D)$ since $\omega_{X/S}\cdot D=0$ by our assumption.  So we need to prove: $$3\omega_{X/S}\cdot (\widetilde{B}+D)\geq 2p_a(\widetilde{B}+D)+1=3+(\widetilde{B}+D)^2+\omega_{X/S}\cdot (\widetilde{B}+D),$$
\ie
$$2\omega_{X/S}\cdot(\widetilde{B}+D)-(\widetilde{B}+D)^2\geq 3.$$
This is true since $\omega_{X/S}\cdot(\widetilde{B}+D)\geq 1$, 
$-(\widetilde{B}+D)^2\geq 0$, and if the left-hand
  side
is equal to $2$, then $\omega_{X/S}\cdot(\widetilde{B}+D)=1$ and 
$(\widetilde{B}+D)^2=0$, which is impossible as 
$(\widetilde{B}+D)^2 +\omega_{X/S}\cdot(\widetilde{B}+D)$ is always even.
\end{proof}

\begin{Prop}[\cite{L}, Theorem 7] \label{Lee} Keep the notation of
  Theorem \ref{main} and suppose $S$ is affine.
Then $\omega_{X/S}^{\otimes n}$ is globally generated for all $n\ge 2$.
\end{Prop} 

\begin{proof} Keep the notation of the above proof. Then it is enough
  to prove that $\omega_{X_s/k}^{\otimes n}$ is globally
generated. By Theorem \ref{CR} (1), it is enough to prove that 
$\omega_{X/S}^{\otimes n}\cdot B\geq 2p_a(B)$ for all $0<B\leq X_s$,
or that 
$$2\omega_{X/S}\cdot B\geq 2+(\omega_{X/S}\cdot B+B^2)$$
which is equivalent to 
$$\omega_{X/S}\cdot B-B^2\geq 2.$$
It is clear that the left-hand side is a non-negative even number, and it is zero only if both $\omega_{X/S}\cdot B=0$ and $B^2=0$. But this is impossible as we pointed out in the proof of Lemma \ref{special}.
\end{proof}

\begin{section}{Proof of the main theorem in the equal-characteristic case}\label{proof 2}

In this section we give an alternative proof of the main theorem under the 
assumption that $S$ is of equal-characteristic (\ie $\cO_S$ contains a field). 
We keep the notation of Section $1$. 
Using  Lemma \ref{special} and the upper
semi-continuity theorem (\cite{H1}, III.12), we find that it is sufficient to prove $\omega_{X'_s/k(s)}^{\otimes n}$ is very ample for $n\geq 3$.  
Note that this conclusion in fact depends only on $X_s$: the pluri-canonical morphism
restricted to the special fibre is exactly that defined by $H^0(X_s, \omega_{X_s/k}^{\otimes n})$, and $X'_{s}=\mathrm{Proj}(\bigoplus\limits_{i} H^0(X_s, \omega_{X_s/k}^{\otimes i}))$ is  canonically determined by $X_s$. 
In particular, to prove the main theorem we are able to interchange $X$ with another minimal arithmetic surface which has the same special fibre.

The next lemma will allow us to reduce to the case when $S$ is a curve over a field.

\begin{Lem}\label{1} Suppose $S=\Spec R$ with $R=k[\![t]\!]$ and $k$ algebraically
closed. Let $s$ be the closed point of $S$. 
Then there exists a minimal fibration $h: Y\rightarrow C$ with $Y$ an
integral projective smooth surface of general type over $k$, $C$ an integral projective smooth 
curve over $k$ of genus $g\geq 2$, and a 
closed point $c\in C$ such that $X_{s}$ is isomorphic to $Y_{c}$ as $k$-schemes. 
\end{Lem}

\begin{proof} 
Let $A$ be the Henselisation of $k[t]_{tk[t]}$.  Then $R$ is the completion of $A$. In particular $k[t]/t^2=R/t^2R= A/t^2A$. Let  $S_1$ and $V$ denote the spectra of $k[t]/t^2$ and of $A$ respectively. Thus $\{s\}=\Spec k$ and $S_{1}$ are considered as closed subschemes of both $S$ and $V$. 

As $f: X\rightarrow S$ is flat and projective, it descends to a  
flat projective scheme $W\to T$ with $T$ integral of finite type over $k$, 
\cite{EGA}, IV.11.2.6. By Artin approximation theorem in \cite{Ar}, there is a morphism $\varphi: V\rightarrow T$, such that the diagram below commutes:
$$
\xymatrix{
S_1\ar@{^{(}->}[r]\ar@{^{(}->}[d]& V \ar[d]^{\varphi}\\
S\ar[r]&T}
$$
We claim that $Z：=W\times_{T} V$ is regular. Indeed, 
$Z_{s}\simeq X_{s}$ as $s$ is a closed subscheme of $S_{1}$ and of 
$X\times_{S} S_{1}=W\times_{T} S_{1}=Z\times _{V} S_{1}$. With this
identification, for any closed point $p\in Z_{s}$, we have 
$T_{Z, p}=T_{Z\times _{V} S_{1}, p}= T_{X\times_{S} S_{1}, p}=T_{X,
  p}$, where $T_{?,?}$ denotes the Zariski tangent space. 
As $\dim_k T_{X, p}=2$ because $X$ is regular, we have $\dim_k
T_{Z,p}\le 2$. By the flatness of $Z/V$, $Z$ is a $2$-dimensional
scheme. So $Z$ is regular. Note also that $Z\to V$ is relatively minimal 
as this property can be checked in the closed fiber and 
$Z_s\simeq X_s$. 

By the construction of $V$, $Z\to V$ descends to a relatively minimal
arithmetic surface $Y_{1}\to C_1$ where $C_1$ is an integral affine
smooth curve over $k$. After compactifying $C_1$ and
  $Y_1$, we find a minimal fibration $Y\to C$ as desired. We take the point $c\in C$ 
to be the image of $s\in V$ in $C_1\subset C$. Finally replacing $C$ by some finite cover that is \'etale at $c$ if necessary, we may assume $g(C)\geq 2$, and consequently $Y$ is of general type (\cite{CZ} Theorem 1.3).
\end{proof}

Note that when $g(C)\geq 1$, the relative canonical model $X/C$ coincides with the canonical model of $X$. 

\begin{Thm}\label{new pluri} Let $S, X, X'$ be as in Section 1. 
Suppose further that $S$ is an integral projective smooth curve of genus 
$\ge 2$ over an algebraically closed field $k$.
Let $K_{X'}$ be a canonical divisor of $X'$ over $k$. Then for any sufficiently 
ample divisor $M$ on $S$, $nK_{X'}+{f'}^*M$ is very ample on $X'$ if $n\geq 3$.
\end{Thm}

\begin{proof}
By Reider's method (see Corollary \ref{corollary}(3)), we just need to prove that 
$H^{1}(X, nK_{X}+f^{*}M-2Z)=0$, where $Z$ is the
 vertical fundamental cycle of the $(-2)$-curves lying above a
singularity of $X'$. 

By our assumption on $M$ (being sufficient ample) and applying spectral sequence, it is sufficient to show $R^{1}f_{*}\mathcal{O}_{X}(nK_{X}+f^{*}M-2Z)=0$ or, equivalently, that $H^{1}(X_{s}, \mathcal{O}_{X}(nK_{X}-2Z)|_{X_{s}})=0$ for any closed point $s\in S$ (note that $\cO_X(f^*M)|_{X_s}$ is trivial). As $\cO_{X}(K_{X})|_{X_s}\simeq \omega_{X/S}|_{X_s}$, by Lemma \ref{special} we only need to consider the case when $s$ is the
image of $Z$ in $S$. Again by 
Lemma~\ref{special}, to prove the vanishing above $s$, it is enough to show that for any divisor $0<A\leq X_{s}$, 
$$((n-1)K_{X}-2Z)\cdot A> A^2.$$ 
Suppose the contrary. Then  
$$2\geq(A+Z)^2+2\geq (n-1)K_{X}\cdot A \geq 2K_{X}\cdot A\ge 0$$ 
(note that $Z^2=-2$). This is impossible:  

(i) 
if $K_{X}\cdot A=0$, then $A$ consists of $(-2)$-curves, in particular $Z\cdot A\leq0$ by the definition of the fundamental cycle, so $(A+Z)^2\leq A^2+Z^2<-2$, contradiction; 

(ii) if $K_{X}\cdot A=1$, then $(A+Z)^2=0$, so $(A+Z)\cdot B=0$ for any
vertical divisor $B$ and thus $A^2=((A+Z)-Z)^2=-2$. This implies that 
$K_{X}\cdot A+A^2=-1$ is odd, contradiction. 
\end{proof} 

Now we can proceed to the proof of our main results.

\begin{proof}[Proof of Theorem~\ref{main} and of Proposition~\ref{Lee}] Similarly to the previous 
section, we can assume $S=\Spec R$ with $R=k[\![t]\!]$ and $k$ algebraically closed. 
Using Lemma~\ref{1}, we can descend $X\to S$ and suppose that $S$ is an 
integral projective smooth curve over $k$ of genus $g\geq 2$. 
Let $M$ be a sufficiently ample divisor on $S$. By Theorem \ref{new pluri}, $nK_{X'}+{f'}^*M$ is very ample for any $n\ge 3$. Therefore 
$\omega_{X'_s/k}^{\otimes n}\simeq \cO_{X'}(nK_{X'}+{f'}^{*}M)|_{X'_{s}}$ is also 
very ample. Similarly, if $n\ge 2$, $\omega_{X_s/k}^{\otimes n}$
is globally generated using Corollary \ref{corollary} (1). 
\end{proof}

\begin{Rem} One can also prove Theorem~\ref{main} under the assumptions of \ref{new pluri} using 
Theorem 1.2 of \cite{CFHR}. Indeed, replacing $C$ by a finite \'etale cover of sufficiently high degree if necessary, we may assume $\chi(\cO_X)\neq 1$, and $g(C)\gg 0$ so that $p_g(X)\geq 2$ (if $\chi(\cO_X)\geq 0$, then 
$p_g(X)\geq q(X)-1\ge g(C)-1\gg 2$; if $\chi(\cO_X)<0$, then $p_g(X)\geq -2\chi(\cO_X)\geq 2$ by \cite{SB}, Lemma 13). 
One then checks that the conditions of \cite{CFHR}, Theorem 1.2 are satisfied. Therefore 
$|nK_{X'}|$ is very ample if $n\geq 3$. In particular $\omega_{X'_s/k}^{\otimes n}\simeq \cO_{X'}(nK_{X'})|_{X'_{s}}$ is also very ample.
\end{Rem}

\end{section}

\begin{section}{Reider's method}

Below is the statement of Reider's method (\cite{R}; \cite{BPV}, p. 176) in any characteristic. 

\begin{Thm}[Reider's Method] \label{Reider}
Let $X$ be an integral projective smooth surface over an algebraically closed field, 
and let $L$ be a nef divisor on $X$. 
\begin{enumerate}
\item Suppose that any vector bundle $E$ of rank $2$ with $\delta(E):= c_{1}^{2}(E)-4c_{2}(E)\geq L^2-4$ is unstable in the sense of Bogomolov (\cite{Bog}; \cite{BPV}, p. 168). If $P$ is a base point of $|K_{X}+L|$, then there is an effective divisor $D$ passing through $P$ such that 
\begin{enumerate}
\item $D\cdot L=0$ and $D^2=-1$; or

\item $D\cdot L=1$ and $D^2=0$.
\end{enumerate}
\item Suppose that any vector bundle $E$ of rank $2$
  with $\delta(E)\geq L^2-8$ is unstable in the sense of Bogomolov. If
  $|K_{X}+L|$ fails to separate $P,Q$ (possibly infinite near), then
  there is an effective divisor $D$ 
passing through $P, Q$ such that 
\begin{enumerate} 
\item $D\cdot L=0$ and $D^2=-2$ or $-1$; or

\item $D\cdot L=1$ and $D^2=0$ or $-1$;  or 

\item $D\cdot L=2$ and $D^2=0$; or

\item $L^2=9$ and $L$ numerically equivalent to $3D$.\qed
\end{enumerate}
\end{enumerate}
\end{Thm}

\begin{Cor}\label{corollary}
Keep the above notation and suppose there exists  a
relatively minimal fibration $f: X\rightarrow S$ whose fibres have arithmetic genus 
$\geq 2$ and such that $g(S)\ge 2$. Then: 
\begin{enumerate}
\item $|nK_{X}+f^{*}M|$ is base point free if $n\geq 2$ and $M$ is sufficiently ample.

\item When $n\geq 3$ and $M$ is sufficiently ample, $|nK_{X}+f^{*}M|$
  is very ample outside the locus of vertical $(-2)$-curves and also
  separates points not connected by vertical $(-2)$-curves.  

\item Keep the hypothesis of (2). Let $f': X'\to S$ be
  the canonical model of $X$ and let $K_{X'}$ be a canonical divisor
  on $X'$. Suppose that for the fundamental cycle
  $Z\subset X$ above any singular point of $X'$ we have 
$$H^1(X, \cO_X(nK_X+f^*M-2Z))=0,$$ 
then $nK_{X'}+{f'}^*M$ is very ample.
\end{enumerate}
\end{Cor}

\begin{proof}
(1) - (2) Let $L=(n-1)K_{X}+f^{*}M$, so 
$$L^2=(n-1)^2K_X^{2}+2(n-1)(2g-2)\deg M\gg 0, \quad \text{if } \deg M\gg 0.$$ 
In characteristic $0$, it is well known that any $E$ with $\delta(E)>0$ is unstable. 
In positive characteristic $p$ case, we apply
\cite{SB}, Theorem 15, which says that $E$ is
semi-stable only if either $pK_{X}^2\geq p/2(\sqrt{K_{X}^2\cdot
  \delta})+2\chi(\mathcal{O}_{X})+p(2p-1)\delta/6$ or $K_{X}^2\geq
\delta$. So anyway, when $M$ is sufficiently ample, the instability
conditions on vector bundles in the above theorem hold. Then by
standard discussions we can prove the conditions (a)(b) in \ref{Reider}(1) will not occur and in \ref{Reider}(2) only condition (a) where $L\cdot D=0, D^2=-2$ can occur, in this case $D$ is a sum of (-2)-curves.  So (1), (2) are proved. 

(3) We have $H^0(X', \cO_{X'}(nK_{X'}+{f'}^*M))=H^0(X, \cO_{X}(nK_{X}+{f}^*M))$ and the map 
$X\to\phi_{|nK_{X}+f^{*}M|}$ factors through $h : X'\to \phi_{|nK_{X}+f^{*}M|}$. 
Under the assumption of (2), $h$ is a homeomorphism and is an
isomorphism outside the singularities of $X'$. In order to prove it is an isomorphism it is sufficient to prove that $|nK_{X'}+{f'}^*M|$ separates tangent spaces of each singularities of $X'$ (see also \cite{Ca} p. 72). Let $x$ be such a singularity and $Z$ be the fundamental cycle of $(-2)$-curves lying above $x$. So what we need is to prove the following sequence is exact:
\begin{equation*}
\begin{split}
0&\rightarrow H^0(X',\mathfrak m_x^2\cO_{X'}(nK_{X'}+{f'}^*M))\rightarrow H^0(X',\mathfrak m_x\cO_{X'}(nK_{X'}+{f'}^*M))\\ &\rightarrow H^0(X', \mathfrak m_x/{\mathfrak m_x^2}\otimes\cO_{X'}(nK_{X'}+{f'}^*M)) \rightarrow 0
\end{split}
\end{equation*} 
where $\mathfrak m_x$ denotes the maximal ideal of $\cO_{X',x}$. It is 
enough to show that  
$$H^1(X', \mathfrak m_x^2\cO_{X'}(nK_{X'}+{f'}^*M))=0.$$ 
Let $\pi$ denote the canonical morphism from $X$ to $X'$, then it is well 
known that $\pi_*\cO_X(nK_X+f^*M-2Z)=\mathfrak m_x^2\cO_{X'}(nK_{X'}+{f'}^*M)$ (\cite{Ba}, Thm. 3.28). For any irreducible component $E$ of the exceptional locus of $X\to X'$, we have
$$E\cdot (nK_X+f^*M-2Z)=nE\cdot K_X-2E\cdot Z=nE \cdot \omega_{X/S}-2E\cdot Z\ge 0$$
by the minimality of $X\to S$ and because $Z$ is a fundamental cycle. 
Therefore it follows from  Lemma \ref{global generated} that 
$R^1\pi_*\cO_X(nK_X+f^*M-2Z)=0$, hence $$H^1(X', \mathfrak m_x^2\cO_{X'}(nK_{X'}+{f'}^*M))=H^1(X, \cO_X(nK_X+f^*M-2Z))=0$$ by assumption.
\end{proof}

\end{section}

\end{document}